\documentclass{amsart}

\usepackage{graphicx}
\usepackage{fancyhdr}
\newtheorem{theorem}{Theorem}[section]
\newtheorem{lemma}[theorem]{Lemma}

\theoremstyle{definition}
\newtheorem{definition}[theorem]{Definition}
\newtheorem{example}[theorem]{Example}

\theoremstyle{remark}
\newtheorem{remark}[theorem]{Remark}

\theoremstyle{proposition}
\newtheorem{proposition}[theorem]{Proposition}

\theoremstyle{corollary}
\newtheorem{corollary}[theorem]{Corollary}

\numberwithin{equation}{section}

\usepackage{indentfirst}
\usepackage{amsmath,amscd}
\usepackage{diagrams}
\usepackage{booktabs}
\usepackage[all]{xy}
\usepackage{setspace}

\newcommand{\pf}{\noindent\begin {proof}}
\newcommand{\epf}{\end{proof}}

\newcommand{\Hom}{\mbox{\rm Hom}}

\def\Im{\mathop{\rm Im}\nolimits}
\def\Ker{\mathop{\rm Ker}\nolimits}
\def\Coker{\mathop{\rm Coker}\nolimits}

\def\mod{\mathop{\rm mod}\nolimits}

\def\add{\mathop{\rm add}\nolimits}

\def\dim{\mathop{\rm dim}\nolimits}

\def\Hom{\mathop{\rm Hom}\nolimits}

\def\lim{\mathop{\underrightarrow{\rm lim}}\nolimits}

\def\End{\mathop{\rm End}\nolimits}

\def\rad{\mathop{\rm rad}\nolimits}

\def\rad{\mathop{\rm rad}\nolimits}

\def\dim{\mathop{\rm dim}\nolimits}
\def\rad{\mathop{\rm rad}\nolimits}
\def\Det{\mathop{\rm Det}\nolimits}
\def\Soc{\mathop{\rm Soc}\nolimits}

\begin{document}

\title{Minimal right determiners of irreducible morphisms in algebras of type ${\mathbb A}_n$}

\author{Xiaoxing Wu}
\address{}
\curraddr{}
\email{}
\thanks{}

\author{Zhaoyong Huang}
\address{Department of Mathematics, Nanjing University, Nanjing 210093, Jiangsu Province, P.R. China}
\email{wuxiaoxing1990@163.com; huangzy@nju.edu.cn}
\thanks{}


\date{}

\dedicatory{}


\begin{abstract}
Let $\Lambda$ be a finite dimensional algebra of type ${\mathbb A}_n$ over a field
with the quiver $Q$ and let $|\Det(\Lambda)|$ be the number of the minimal right determiners
of all irreducible morphisms between indecomposable left $\Lambda$-modules. If $\Lambda$ is a path algebra, then we have
$$|\Det(\Lambda)|=
\begin{cases}
2n-2,   &\mbox{if $p=0$;}\\
2n-p-1, &\mbox{if $p\geq 1$,}
\end{cases}$$
where $p=|\{i\mid i$ is a source in $Q$ with $2\leq i\leq n-1\}|$. If $\Lambda$ is a bound quiver algebra,
then we have
$$
|\Det(\Lambda)|=
\begin{cases}
2n-2, &\mbox{if $r=1$;}\\
2n-p-q-1, &\mbox{if $r\geq 2$,}
\end{cases}
$$
where $q$ is the number of non-zero sink ideals of $\Lambda$ and
$r=|\{i\mid i$ is a sink in $Q$ with $1\leq i\leq n\}|$.
\end{abstract}

\maketitle

\setlength{\baselineskip}{14pt}

\section{Introduction}

In order to classify a certain class of morphisms ending with the same object as well as to generalize almost split morphisms,
Auslander introduced in [2] the notions of functors and morphisms determined by objects.
``Unfortunately, $\cdots$ for over 30 years, Auslander's powerful results did not gain the attraction they deserve" ([11, p.409]).
Recently, some progress on this topic has been made by some authors, see [7--11].

Let $\Lambda$ be an artin algebra and $\mod \Lambda$ the category of finitely generated left $\Lambda$-modules.
Ringel gave in [10, Theorem 1] a formula to calculate a right determiner of a morphism in $\mod \Lambda$,
which is a correction of [3, Theorem 2.6]; and then he reproved in [11, Theorem 3.4] a formula, originally
due to Auslander, Reiten and Smal$\o$ [4], for calculating the minimal right determiner of a morphism in $\mod \Lambda$.
Later on, it was proved in [8, 7] that
for a commutative artin ring $k$, a Hom-finite $k$-linear additive skeletally small category $\mathcal{C}$
is a dualizing $k$-variety if and only if $\mathcal{C}$ has determined morphisms; and a Hom-finite $k$-linear abelian category
$\mathcal{C}$ has Serre duality if and only if $\mathcal{C}$ has right determined epimorphisms and left determined monomorphisms.

Beyond the results mentioned above, it is interesting to determine concretely
the minimal right determiners of a certain class of irreducible morphisms. Certainly, it is rather difficult in general.
The aim of this paper is to determine the minimal right determiners of all irreducible morphisms between indecomposable
modules over a finite dimensional algebra of type $\mathbb{A}_n$.
The paper is organized as follows.

In Section 2, we give some terminology and some preliminary results. In particular, we give a formula for determining
the minimal right determiner of an irreducible morphism between indecomposable left $\Lambda$-modules.
We use $\Det(\Lambda)$ to denote
the set of the minimal right determiners of all irreducible morphisms between indecomposable left $\Lambda$-modules.
For a set $S$, we use $|S|$ to denote the cardinality of $S$.

Let $\Lambda$ be a finite dimensional algebra of type $\mathbb{A}_n$ over a field.
In Section 3, we first prove that for an irreducible morphism $f$ between indecomposable left $\Lambda$-modules,
if $f$ is epic, then the minimal right determiner of $f$
is the end term in an almost split sequence with a unique middle term;
if $f$ is monic, then the minimal right determiner of $f$ is indecomposable projective.
This result is crucial in the sequel.
Let $Q$ be the quiver of $\Lambda$ and $p=|\{i\mid i$ is a source in $Q$ with $2\leq i\leq n-1\}|$.
One of our main results is the following

\begin{theorem} \label{1.1}
If $\Lambda$ is a path algebra, then we have
$$|\Det(\Lambda)|=
\begin{cases}
2n-2,   &\mbox{if $p=0$;}\\
2n-p-1, &\mbox{if $p\geq 1$.}
\end{cases}$$
\end{theorem}

After introducing the notion of sink ideals of $\Lambda$, we get the following

\begin{theorem} \label{1.2}
If $\Lambda$ is a bound quiver algebra, then we have
$$
|\Det(\Lambda)|=
\begin{cases}
2n-2, &\mbox{if $r=1$;}\\
2n-p-q-1, &\mbox{if $r\geq 2$,}
\end{cases}
$$
where $q$ is the number of non-zero sink ideals of $\Lambda$
and $r=|\{i\mid i$ is a sink in $Q$ with $1\leq i\leq n\}|$.
\end{theorem}

The proofs of these two theorems are constructive, from which we can determine $\Det(\Lambda)$.
Finally, in Section 4, we give some examples to illustrate the obtained results.

\section{Preliminaries}

Throughout this paper, $\Lambda$ is a finite dimensional algebra over a field $K$
with the quiver $Q$ and $\mod \Lambda$ is the category of finitely generated left $\Lambda$-modules.
For an arrow $\alpha$ in $Q$, $s(\alpha)$ and $e(\alpha)$ are the starting and end points of $\alpha$, respectively.
We use $P(i)$, $I(i)$ and $S(i)$ to denote the indecomposable projective, injective
and simple modules corresponding to the vertex $i$, respectively. Let $M$ be in $\mod \Lambda$. We use
$\add_{\Lambda}M$ to denote the full subcategory of $\mod \Lambda$ consisting of direct summands of
finite direct sums of copies of $M$; and use $\dim M$ and $\textbf{dim}M$ to denote the dimension
as a $K$-vector space and the dimension vector of $M$ respectively. For a set $S$, we use $|S|$
to denote the cardinality of $S$.

The original definition of morphisms determined by objects in [2] is based on the notion
of subfunctors determined by objects. However, in the relevant papers, ones prefer the following definition
since it is easier to understand.

\begin{definition} {\rm ([10, 11])}
For a module $C\in \mod \Lambda$, a morphism $f\in \Hom_{\Lambda}(X,Y)$ is said to be {\bf right determined}
by $C$ (simply {\bf $C$-right determined})
if the following condition is satisfied: given for any $f'\in\Hom_{\Lambda}(X',Y)$ such that $f'\phi$
factors through $f$ for all $\phi\in\Hom_{\Lambda}(C,X')$, then $f'$ factors through $f$; that is,
in the following diagram, if there exists $\phi'\in\Hom_{\Lambda}(C,X)$ such that $f'\phi=f\phi'$,
then there exists $h\in\Hom_{\Lambda}(X',X)$ such that $f'=fh$.
$$\xymatrix{C \ar[r]^{\phi} \ar@{=}[d] & X' \ar[r]^{f'} \ar@{-->}[d]^{h} & Y \ar@{=}[d] \\
C \ar@{-->}[r]^{\phi'} & X \ar[r]^{f} & Y.}$$
In this case, $C$ is called a {\bf right determiner} of $f$.
\end{definition}

Auslander proved in [2] that a morphism $f\in \Hom_{\Lambda}(B,C)$ is right almost split
if and only if the following conditions are satisfied: (a) the endomorphism ring $\End_{\Lambda}(C)$ is local;
(b) $\Im\Hom_{\Lambda}(C,f)$ is the unique maximal ideal of $\End_{\Lambda}(C)$; and
(c) $f$ is $C$-right determined. It means that a morphism right determined by a module is a generalization
of a right almost split morphism.

\begin{definition}
([10, p.984])
Given a morphism $f\in\Hom_{\Lambda}(B,C)$ with $B=B_1\oplus B_2$ such that $B_1 \subseteq \Ker f$
and $f|_{B_2}$ is right minimal, then we call $\Ker f|_{B_2}$ the {\bf intrinsic kernel} of $f$.
\end{definition}

\begin{definition} {\rm ([11, p.418])}
An indecomposable projective module $P\in \mod \Lambda$ is said to {\bf almost factor through}
$f\in \Hom_{\Lambda}(M,N)$ provided that there exists a commutative diagram of the following form
$$\xymatrix{
\rad P \ar[d] \ar[r]^{i}
& P \ar[d]^{h}  \\
M \ar[r]^{f} & N,}
$$
where $i$ is the inclusion map and $\rad P$ is the radical of $P$, such that $\Im h$ is not contained in $\Im f$.
\end{definition}

The following is the determiner formula.

\begin{theorem}
Let $f$ be a morphism in $\mod \Lambda$. Let $C(f)$ be the direct sum of the indecomposable
modules of the form $\tau^{-1}K$, where $\tau$ denotes the Auslander-Reiten translation
and $K$ is an indecomposable direct summand of
the intrinsic kernel of $f$ and of the indecomposable projective modules which
almost factor through $f$, one from each isomorphism class. Then we have
\begin{enumerate}
\item[(1)] $f$ is right $C$-determined if and only if $C(f)\in\add_{\Lambda}C$.
\item[(2)] If $f$ is irreducible, then $C(f)$=$\tau^{-1}\Ker f\oplus(\oplus P_i)$,
where all $P_i$ are pairwise non-isomorphic indecomposable projective modules almost factoring through $f$.
\end{enumerate}
\end{theorem}

\begin{proof}
(1) See [11, Theorem 3.4] (cf. [10, Theorem 2] and [4, Corollary XI.2.3]).

(2) Note that the intrinsic kernel of any morphism is a direct summand of its kernel ([11, p.419]).
So the intrinsic kernel of an irreducible morphism coincides with its kernel and the assertion follows.
\end{proof}

The first assertion in this theorem suggests to call $C(f)$ the {\bf minimal right determiner} of $f$ ([10, 11]).
We use $\Det(\Lambda)$ to denote the set of the minimal right determiners of all irreducible
morphisms between indecomposable modules in $\mod \Lambda$.

\begin{lemma}
Let $f$: $B\rightarrow C$ be a minimal right almost split morphism. Then $C(f)=C$.
\end{lemma}

\begin{proof}
Because $f$ is right $C$-determined and $C$ is indecomposable, the assertion follows.
\end{proof}

We say that $\Lambda$ is of {\bf type $\mathbb{A}_n$} if $Q$ is
of type $\mathbb{A}_n$, that is, the underlying unoriented graph
of $Q$ is the Dynkin diagram of type $\mathbb{A}_n$:
$$1\frac{\alpha_1}{}2 \frac{\alpha_2}{} 3\frac{\alpha_3}{}\cdots \frac{\alpha_{n-2}}{} n-1\frac{\alpha_{n-1}}{}n.$$
A subquiver of $Q$ between $i$ and $j$ with $i<j$ is denoted by $<i,j>$.

\begin{lemma}
Let $\Lambda$ be of type $\mathbb{A}_n$. Let $P\in \mod \Lambda$ be indecomposable projective and $0\neq f\in \Hom_{\Lambda}(M,N)$
monic with $M,N\in\mod \Lambda$ indecomposable. If $P$ almost factors through $f$, then
$$\dim \Hom_{\Lambda}(P,M)=0\ and \ \dim \Hom_{\Lambda}(P,N)=1.$$
\end{lemma}


\begin{proof}
By [12, Section 3.1.4.1], we have that $\dim\Hom_{\Lambda}(X,Y)\leq 1$ for any indecomposable
modules $X,Y\in \mod \Lambda$. Let $P\in \mod \Lambda$ be indecomposable projective
almost factoring through $f\in \Hom_{\Lambda}(M,N)$
with $M,N$ indecomposable. Then $\dim\Hom_{\Lambda}(P,N)\neq 0$ and $\dim\Hom_{\Lambda}(P,N)=1$.

Let $0\neq h\in \Hom_{\Lambda}(P,N)$ be as in Definition 2.3.
If $\dim \Hom_{\Lambda}(P,M)=1$, then there exists $0\neq g\in\Hom_{\Lambda}(P,M)$.
Because $f$ is monic by assumption, we have $fg\neq 0$.
Because $h,fg \in \Hom_{\Lambda}(P,N)$ and $\dim\Hom_{\Lambda}(P,N)=1$,
we have $\Im h=\Im fg\subseteq \Im f$, which is a contradiction.
Thus $\dim \Hom_{\Lambda}(P,M)=0$.
\end{proof}





The following observation might be known.

\begin{lemma}
Let
$$0\to L  \buildrel {\binom {f_1} {g_1}} \over \longrightarrow
M_1\oplus M_2 \buildrel {(g_2, f_2)} \over \longrightarrow M \to 0$$
be an almost split sequence with $M_1$ and $M_2$ indecomposable.
Then $f_1, f_2$ (resp. $g_1, g_2$) are either both epic or both monic.
\end{lemma}

\begin{proof}
We have
$$\textbf{dim} L + \textbf{dim} N=\textbf{dim} M_1 + \textbf{dim} M_2.$$
To prove that $f_1, f_2$ are either both epic or both monic, it suffices to prove that
if $f_1$ is monic, then so is $f_2$. The other assertions follow similarly.
It is well known that $f_i$ is either monic or epic but non-isomorphic for $i=1,2$
(see [1, Section IV.1]). Let $f_1$ be monic. If $f_2$ is epic, then we have
$$\textbf{dim} \Coker f_1=\textbf{dim} M_1-\textbf{dim} L,$$
$$\textbf{dim} \Ker f_2=\textbf{dim} M_2-\textbf{dim} N.$$
So $\textbf{dim} \Coker f_1+\textbf{dim} \Ker f_2=0$ and $\Coker f_1=0= \Ker f_2$,
and hence both $f_1$ and $f_2$ are isomorphic, which is a contradiction.

Similarly, we get that $g_1, g_2$ are either both epic or both monic.
\end{proof}

Recall from [6] that $A=KQ^{op}/I$ with $I$ an ideal of $KQ^{op}$
is called a {\bf string algebra} if the following conditions are satisfied:
\begin{enumerate}
\item[(1)]  Any vertex of $Q$ is starting point of at most two arrows.
\item[($1^{op}$)] Any vertex of $Q$ is end point of at most two arrows.
\item[(2)]  Given an arrow $\beta$, there is at most one arrow $\gamma$ with $s(\beta)=e(\gamma)$
and $\beta \gamma \notin I$.
\item[($2^{op}$)] Given an arrow $\gamma$, there is at most one arrow $\beta$ with $s(\beta)=e(\gamma)$
and $\beta \gamma \notin I$.
\item[(3)] Given an arrow $\beta$, there is some bound $n(\beta)$ such that any path $\beta_1 \cdots \beta_{n(\beta)}$
with $\beta_1 = \beta$ contains a subpath in $I$.
\item[($3^{op}$)] Given an arrow $\beta$, there is some bound $n^{\prime}(\beta)$ such that any path
$\beta_1 \cdots \beta_{n^{\prime}(\beta)}$ with $\beta_{n^{\prime}(\beta)} = \beta$ contains a subpath in $I$.
\end{enumerate}

Let $\Lambda$ be of type $\mathbb{A}_n$. Then it is trivial that $\Lambda$ is a string algebra.
By [12, Section 3.1.1], there are only two types of almost split sequences in
$\mod\Lambda$, that is, the middle term in an almost split sequence is indecomposable or is a direct
sum of two indecomposable modules:
$$0\to L \to M \to N \to 0\eqno{(2.1)}$$
and
$$0\to L \to M_1\oplus M_2 \to N \to 0.\eqno{(2.2)}$$

\begin{lemma}
Let $\Lambda$ be of type $\mathbb{A}_n$. Then we have
\begin{enumerate}
\item[(1)] The only almost split sequences in $\mod \Lambda$ having a unique middle term are those of the form
$$0 \to U(\beta) \to N(\beta) \to V(\beta) \to 0$$
with $\beta$ an arrow of $Q$.
\item[(2)]
The number of almost split sequences of type (2.1) in $\mod \Lambda$ is $n-1$.
\end{enumerate}
\end{lemma}

\begin{proof}
The assertion (1) is a special case of [6, p.174, Corollary], and the assertion (2) is an immediate consequence of (1).
\end{proof}

We recall some notions about string algebras from [5, 6].

Let $\Lambda=KQ/I$ with $I$ an admissible ideal of $KQ$ and let $\beta$ be an arrow of $Q$. We denote by
$\beta^{-1}$ a {\bf formal inverse} of $\beta$
with $s(\beta^{-1})=e(\beta)$ and $e(\beta^{-1})=s(\beta)$, and write $(\beta^{-1})^{-1}=\beta$. We form `paths'
$c_1 \cdots c_n$ of length $n \geq 1$ where all the $c_i$ are of the form $\beta$ or $\beta^{-1}$ and $s(c_i)=e(c_{i+1})$.
Define $(c_1 \cdots c_n)^{-1}=c_n^{-1} \cdots c_1^{-1}$ and $s(c_1 \cdots c_n)=s(c_n)$, $e(c_1 \cdots c_n)=e(c_1)$.
A path $c_1 \cdots c_n$ of length $n \geq 1$ is called a {\bf string} if $c_{i+1} \neq c_i^{-1}$ for any $1 \leq i \leq n-1$,
and neither subpath $c_i c_{i+1} \cdots c_{i+t}$ nor its inverse belongs to the ideal $I$. Also the two strings of length 0
is defined just as the trivial path $\epsilon_i$ at each vertex $i$ and its inverse.
For example, for the quiver
$$\xymatrix{1 \ar[r]^{\alpha _1} &2 \ar[r]^{\alpha _2} &3  &4\ar[l]_{\alpha _3}}$$
with $\alpha_2 \alpha_1=0$, all the strings are $\epsilon_1$, $\epsilon_2$, $\epsilon_3$, $\epsilon_4$, $\alpha _1$, $\alpha _2$,
$\alpha _3$, $\alpha _2 \alpha _1$, $\alpha _3^{-1} \alpha _2$ and all their inverses.

Let $\mathcal{S}$ be the set of all strings. We say that a string $\omega$ {\bf starts (resp. ends) on a peak} if there exists no arrow
$\alpha$ such that $\omega \alpha \in \mathcal{S}$ (resp. $\alpha ^{-1} \omega \in \mathcal{S}$); similarly, a string
{\bf starts (resp. ends) in a deep} if there exists no arrow $\beta$ such that $\omega \beta^{-1} \in \mathcal{S}$
(resp. $\beta \omega \in \mathcal{S}$).
A string $\omega=\alpha_1 \cdots \alpha_n$ is called {\bf direct} if all $\alpha_i$ are arrows, and called {\bf inverse}
if its inverse is direct. Strings of length 0 are both direct and inverse. For example, in the linearly oriented quiver
of type $\mathbb{A}_n$, all the arrows end on a peak and start in a deep since all the strings here are either inverse or direct.
Dually, all the inverses of arrows start on a peak and end in a deep.

For every arrow $\alpha$ in $Q$, let $N_{\alpha}=U_{\alpha} \alpha V_{\alpha}$ be the unique string with $U_{\alpha}$ and $V_{\alpha}$
both inverse and $N_{\alpha}$ starts in a deep and ends on a peak. By [5, Remarks 3.2(1)], there exists an almost split sequence
with an indecomposable middle term:
$$\xymatrix{
0 \ar[r] &M(U_{\alpha}) \ar[r] &M(N_{\alpha}) \ar[r] &M(V_{\alpha}) \ar[r] &0
}$$
for every arrow $\alpha$, and all almost split sequences of this type are constructed in this way.
Here $M(\omega)$ denotes the indecomposable module corresponding to the string $\omega$. Note that $M(\omega)$ and
$M(\omega^{-1})$ are always isomorphic.

Let $\Lambda$ be an algebra of type $\mathbb{A}_n$. We denote all the $n-1$ almost split sequences of type (2.1) by
$$\xymatrix{
0 \ar[r] &M(U_{\alpha_i}) \ar[r] &M(N_{\alpha_i}) \ar[r] &M(V_{\alpha_i}) \ar[r] &0
}$$
with $1 \leq i \leq n-1$. It is known from [12, Section 8.3.1] that all indecomposable representations of $Q$ can be expressed
in the form $[a,b]$ with $1 \leq a \leq b \leq n$, which means as a representation
by putting 1-dimensional vector spaces between $a$ and $b$ (including $a$ and $b$)
and zero vector spaces in other positions with linear maps zero or identity:
$$0 \mbox{-----} 0 \mbox{-----}\cdots \mbox{-----} 0 \mbox{-----} K \mbox{-----} K \mbox{-----} \cdots \mbox{-----}
K \mbox{-----} K \mbox{-----} 0 \mbox{-----} \cdots \mbox{-----} 0.$$
Thus, for any string $\omega$, $M(\omega)=[s(\omega),e(\omega)]$ or $[e(\omega),s(\omega)]$.

\section{Main results}

In this section, let $Q$ be of type $\mathbb{A}_n$ and $\Lambda=KQ$ or $\Lambda=KQ/I$
with $I$ an admissible ideal of $KQ$. All morphisms considered are irreducible morphisms between
indecomposable modules in $\mod \Lambda$.

We begin with the following

\begin{proposition}
\begin{enumerate}
\item[]
\item[(1)]
\ \ \ $\{\Coker f\mid f$ is monic$\}$\\
$=\{\Coker f\mid$ for any monomorphism $f:0 \rightarrow X \rightarrow P(i)\}$\\
$=\{M(V_{\alpha_i})\}_{i=1}^n$.
\item[(2)]
\ \ \ $\{\Ker g\mid g$ is epic$\}$\\
$=\{\Ker g\mid$ for any epimorphism $g:I(i) \rightarrow Y \rightarrow 0\}$\\
$=\{M(U_{\alpha_i})\}_{i=1}^n$.
\end{enumerate}
\end{proposition}

\begin{proof}
By [6, p.174--175], we have
$\{\Coker f\mid$ for any monomorphism $f: 0 \rightarrow X \rightarrow P(i)\}=\{M(V_{\alpha_i})\}_{i=1}^n$.

Because $\Lambda$ is of type $\mathbb{A}_n$, there exist $n-1$ irreducible morphisms in the form
$f_i: 0 \rightarrow X \rightarrow P(i)$ with $1 \leq i \leq n-1$, and the $n-1$ almost split sequences
of type (2.1) are as follows.
$$\xymatrix{
0 \ar[r] &\tau \Coker f_i \ar[r] & M_i \ar[r] & \Coker f_i \ar[r] & 0.
}$$

Let $\omega$ be a string not starting on a peak. Let $\beta_0$, $\beta_1$, $\cdots$, $\beta_n$ be arrows such that
$\omega_h=\omega \beta_0 \beta_1^{-1} \cdots \beta_r^{-1}$ is a string starting in a deep.
Then the canonical embedding $f: M(\omega) \rightarrow M(\omega_h)$ is irreducible ([6, p.166, Lemma]).
Thus $\Coker f= M(\omega_h)/M(\omega)\cong M(\beta_1^{-1} \cdots \beta_r^{-1})\cong M(V_{\beta_0})$.

Let $\omega$ be a string not ending on a peak.
Let $\beta_0$, $\beta_1$, $\cdots$, $\beta_n$ be arrows such that
${}_h\omega= \beta_r \cdots \beta_1  \beta_0^{-1} \omega$ is a string ending in a deep. Then the canonical embedding
$f: M(\omega) \rightarrow M({}_h \omega)$ is irreducible ([6, p.168, Lemma]). Thus $\Coker f= M({}_h\omega)/M(\omega)
\cong$ $M(\beta_r \cdots \beta_1)\cong M(\beta_1^{-1} \cdots \beta_r^{-1})\cong M(V_{\beta_0})$.

Then by [6, p.172, Proposition], we have $\{\Coker f\mid f$ is monic$\}=\{M(V_{\alpha})\}_{i=1}^n$.

(2) It is dual to (1).
\end{proof}

\begin{corollary}
Let $f$ be monic. Then $\Soc(\Coker f)$ is simple
and $C(f)=P(i)$ for some $1\leq i\leq n$.
\end{corollary}

\begin{proof} By Proposition 3.1(1), there exists a unique arrow $\alpha_i$ such that $\Coker f=M(V_{\alpha_i})$.
Here the string $V_{\alpha_i}$ is inverse and by [5, p.169], $M(V_{\alpha_i})$ is uniserial and admits a simple socle.
So we may assume $\Soc(\Coker f)=S(i)$ for some $1\leq i\leq n$.
By [10, Theorem 1], $f$ is right $\tau^{-1} \Ker f \oplus P(S(i))$-determined, where $P(S(i))(=P(i))$ is
the projective cover of $S(i)$. Because $f$ is monic, $f$ is right $P(i)$-determined and $C(f)=P(i)$.
\end{proof}

\begin{remark}
If $f$ is epic, then $C(f)=\tau^{-1}\Ker f$ by Theorem 2.4(2); if $f$ is monic,
then $C(f)=P(i)$ for some $1\leq i\leq n$ by Corollary 3.2. Thus $C(f)$ is always indecomposable.
\end{remark}

\begin{corollary}
\begin{enumerate}
\item[]
\item[(1)]
Let $f_1$ and $f_2$ be monic. Then
$C(f_1)=C(f_2)=P(i)$ if and only if $\Coker f_1$ and $\Coker f_2$ share the same socle $S(i)$.
\item[(2)]
Let $g_1$ and $g_2$ be epic. Then
$C(g_1)=C(g_2)$ if and only if $\Ker g_1 = \Ker g_2$.
\end{enumerate}
\end{corollary}

\begin{proof}
(1) It follows from Corollary 3.2.

(2) Note that $\Ker g_1$ and $\Ker g_2$ are non-zero non-injective indecomposable modules. By [1, Proposition IV.2.10],
we have that $\tau^{-1} \Ker g_1 = \tau^{-1} \Ker g_2$ if and only if $\Ker g_1 = \Ker g_2$.
\end{proof}

Note that an irreducible morphism is either a proper monomorphism or a proper epimorphism.
We have the following descriptions of irreducible morphisms from the viewpoint of minimal right determiners.
It is crucial in the sequel.

\begin{theorem}
\begin{enumerate}
\item[]
\item[(1)]
For any morphism $f$ in $\mod \Lambda$, the following statements are equivalent.
\begin{enumerate}
\item[(1.1)] $f$ is monic.
\item[(1.2)] There exists a monomorphism $f_1:X \rightarrow P(i)$ with $X$ indecomposable
such that $C(f)=C(f_1)=P(j)$ for some $1\leq j \leq n$.
\item[(1.3)] There exists an almost split sequence
$$0 \to M(U_{\alpha_i}) \buildrel {f_2}\over \longrightarrow  M(N_{\alpha_i}) \to M(V_{\alpha_i}) \to 0$$
in $\mod \Lambda$ such that $C(f)=C(f_2)=P(j)$ for some $1\leq j \leq n$.
\end{enumerate}
\end{enumerate}
\begin{enumerate}
\item[(2)]
For any morphism $g$ in $\mod \Lambda$, the following statements are equivalent.
\begin{enumerate}
\item[(2.1)] $g$ is epic.
\item[(2.2)] There exists a unique epimorphism $g_1:I(i) \rightarrow Y$ with $Y$ indecomposable such that $C(g)=C(g_1)$.
\item[(2.3)] There exists a unique almost split sequence
$$0 \to M(U_{\alpha_i}) \to M(N_{\alpha_i}) \buildrel {g_2}\over \longrightarrow M(V_{\alpha_i}) \to 0$$
in $\mod \Lambda$ such that $C(g)=C(g_2)=\tau^{-1}\Ker g_2 =M(V_{\alpha_i})$, which is
a non-zero non-projective indecomposable module.
\end{enumerate}
\end{enumerate}
\end{theorem}

\begin{proof}
(1) By [10, Section 3], $C(f)$ exists and is non-zero.

$(1.1)\Rightarrow (1.3)$  By Proposition 3.1(1), there exists an almost split sequence
$$0 \to M(U_{\alpha_i}) \buildrel {f_2}\over \longrightarrow  M(N_{\alpha_i}) \to M(V_{\alpha_i}) \to 0$$
in $\mod \Lambda$ such that $\Coker f = \Coker f_2$. Then by Corollary 3.4(1), we have $C(f)=C(f_2)=P(j)$
for some $1\leq j \leq n$.

$(1.3)\Rightarrow (1.2)$ By Proposition 3.1(1), there exists a monomorphism $f_1:X \rightarrow P(i)$ with $X$ indecomposable
such that $\Coker f_2=\Coker f_1$. Then by Corollary 3.4(1), $C(f_2)=C(f_1)=P(j)$
for some $1\leq j \leq n$.

$(1.2)\Rightarrow (1.1)$ If $C(f)=P(j)$, then $f$ is not epic, and so it is monic.

(2) By [10, Section 3], $C(g)$ exists and is non-zero.

$(2.1)\Rightarrow (2.3)$
By Proposition 3.1(2), there exists an almost split sequence
$$0 \to M(U_{\alpha_i}) \to M(N_{\alpha_i}) \buildrel {g_2}\over \longrightarrow M(V_{\alpha_i}) \to 0$$
in $\mod \Lambda$ such that $\Ker g = \Ker g_2$. Then by Corollary 3.4(2) and Lemma 2.5, we have $C(g)=C(g_2)=M(V_{\alpha_i})$.
The uniqueness also follows from Corollary 3.4(2) and Lemma 2.5.

$(2.3)\Rightarrow (2.2)$
By Proposition 3.1(2), there exists an epimorphism $g_1:I(i) \rightarrow Y$ with $Y$ indecomposable
such that $\Ker g_1=\Ker g_2$. Then by Corollary 3.4(2), we have $C(g_1)=C(g_2)$.

If there exists another epimorphism $g_0: I(j) \rightarrow Z$ with $Z$ indecomposable and $j\neq i$
such that $C(g)=C(g_0)$, then $\Ker g_1=\Ker g_0$ by Corollary 3.4(2), and so there exists $h_1:I(i)
\rightarrow I(j)$ and $h_2:I(j) \rightarrow I(i)$ such that $f_2=h_1 i_1$ and $f_1=h_2 i_2$.
It implies that both $h_1$ and $h_2$ are non-zero, and hence $\dim \Hom_{\Lambda}(I(i), I(j))=1$ and
$\dim \Hom_{\Lambda}(I(j), I(i))=1$ by [12, Section 3.1.4.1].
It is a contradiction.

$(2.2)\Rightarrow (2.1)$ If $C(g)$ is not projective, then $\Ker g\neq 0$ and $g$ is epic.
\end{proof}

In the above theorem, unlike (2.2) and (2.3), neither (1.2) nor (1.3) possesses the uniqueness as shown in following example.
\begin{example}
Let $Q$ be the quiver
$$1 \leftarrow 2\rightarrow 3 \leftarrow 4\rightarrow  5 \leftarrow 6$$ and $\Lambda=KQ$. Then the Auslander-Reiten quiver of $\mod \Lambda$
is as follows.
$$
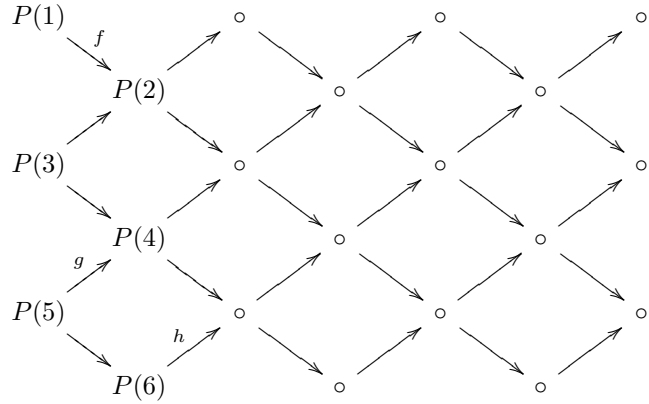
\xymatrix@-12pt{
\makebox[20pt]{$P(1)$} \ar[dr]^{f} && \makebox[20pt]{$\circ$} \ar[dr] && \makebox[20pt]{$\circ$} \ar[dr] && \makebox[20pt]{$\circ$} \\
& \makebox[20pt]{$P(2)$} \ar[ur]\ar[dr] && \makebox[20pt]{$\circ$} \ar[ur]\ar[dr] && \makebox[20pt]{$\circ$} \ar[ur]\ar[dr] \\
\makebox[20pt]{$P(3)$}\ar[ur]\ar[dr] && \makebox[20pt]{$\circ$} \ar[ur]\ar[dr] && \makebox[20pt]{$\circ$} \ar[ur]\ar[dr] && \makebox[20pt]{$\circ$} \\
& \makebox[20pt]{$P(4)$} \ar[ur]\ar[dr] && \makebox[20pt]{$\circ$} \ar[ur]\ar[dr] && \makebox[20pt]{$\circ$} \ar[ur]\ar[dr] \\
\makebox[20pt]{$P(5)$}\ar[ur]^{g}\ar[dr] && \makebox[20pt]{$\circ$} \ar[ur]\ar[dr] && \makebox[20pt]{$\circ$} \ar[ur]\ar[dr] && \makebox[20pt]{$\circ$} \\
&\makebox[20pt]{$P(6)$} \ar[ur]^{h} && \makebox[20pt]{$\circ$} \ar[ur] && \makebox[20pt]{$\circ$} \ar[ur]
}$$
and we have $C(f)=C(g)=C(h)=P(3)$.
\end{example}

\newpage
By Theorem 3.5(2), we immediately get the following

\begin{corollary}
We have
$$\{C(f)\mid f\ \text{is an epic irreducible morphism in}\ \mod \Lambda\}=\{M(V_{\alpha_i})\mid 1\leq i \leq n-1\}.$$
\end{corollary}

\begin{corollary}
$|\Det(\Lambda)| \leq 2n-2$.
\end{corollary}

\begin{proof}
By Lemma 2.8, there exist $n-1$ almost split sequences of type (2.1) in $\mod \Lambda$.
By Theorem 3.5, to calculate the minimal right determiners of all irreducible morphisms between indecomposable modules,
it suffices to calculate the minimal right determiners of the $2n-2$ morphisms in the $n-1$ almost split sequences of type (2.1).
The assertion follows.
\end{proof}

To get the general formula about $|\Det(\Lambda)|$, we need more preparations.

\begin{proposition}
If $i$ with $2 \leq i \leq  n-1$ is a source in $Q$, then $P(i)$ is not a minimal right determiner
of any irreducible morphism.
\end{proposition}

\begin{proof}
It is trivial that $P(i)$ can not almost factor through any epic irreducible morphism.
In the following, we prove that $P(i)$ can not almost factor through any monic irreducible morphism.

We claim that $i$ with $2 \leq i \leq n-1$ is a source in $Q$ if and only if $\rad P(i)$ is non-zero
and not indecomposable. First we have that $P(i)=[s,t]$ with $s < i <t $ and $\rad P(i)=P(i)/S(i)=[s,i-1] \oplus [i+1,t]$
is not indecomposable. If $i$ is not a source, then $i$ can be described in three cases:
\begin{enumerate}
\item[(1)] $\cdots  i-1 \to  i \to i+1 \cdots$: $P(i)=[i,t]$ and $\rad P(i)=[i+1,t]$ is indecomposable.
\item[(2)] $\cdots i-1  \leftarrow i \leftarrow i+1 \cdots$: $P(i)=[s,i]$ and $\rad P(i)=[s,i-1]$ is indecomposable.\\
\item[(3)] $\cdots  i-1 \to i \leftarrow i+1 \cdots$: $P(i)=S(i)$ and $\rad P(i)=0$.
\end{enumerate}
The claim is proved.

Consider the maximal sectional paths with respect to $P(i)$:
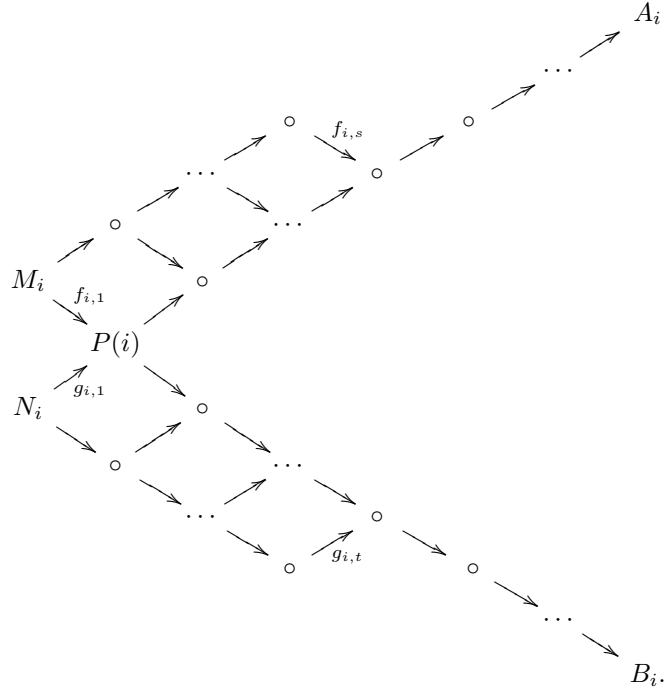
$$\xymatrix@-15pt{
&&&&&&& \makebox[18pt]{$A_i$}          \\
&&&&&& \makebox[18pt]{$\cdots$} \ar[ur] \\
&&& \makebox[18pt]{$\circ$} \ar[dr]^{f_{i,s}} && \makebox[18pt]{$\circ$} \ar[ur]\\
&& \makebox[18pt]{$\cdots$} \ar[dr]\ar[ur] && \makebox[18pt]{$\circ$} \ar[ur] \\
& \makebox[18pt]{$\circ$} \ar[dr]\ar[ur] && \makebox[18pt]{$\cdots$} \ar[ur]\\
\makebox[18pt]{$M_i$}  \ar[dr]^{f_{i,1}}\ar[ur] && \makebox[18pt]{$\circ$} \ar[ur]\\
& \makebox[18pt]{$P(i)$} \ar[dr]\ar[ur] \\
\makebox[18pt]{$N_i$}  \ar[dr]\ar[ur]_{g_{i,1}} && \makebox[18pt]{$\circ$} \ar[dr]\\
& \makebox[18pt]{$\circ$} \ar[dr]\ar[ur] && \makebox[18pt]{$\cdots$} \ar[dr]\\
&& \makebox[18pt]{$\cdots$} \ar[dr]\ar[ur] && \makebox[18pt]{$\circ$} \ar[dr]\\
&&& \makebox[18pt]{$\circ$} \ar[ur]_{g_{i,t}} && \ \makebox[18pt]{$\circ$} \ar[dr] \\
&&&&&& \makebox[18pt]{$\cdots$} \ar[dr] \\
&&&&&&& \makebox[18pt]{$B_i$.}
}$$
Then by Lemmas 2.6 and 2.7, it is easy to check that the possible morphisms
that $P(i)$ might almost factor through are all the $f_{i,l} (1 \leq l \leq s)$ and $g_{i,m} (1 \leq m \leq t)$.
However, in the following we prove that $P(i)$ can not almost factor through any one of them.
Consider the following diagram:
$$\xymatrix{
M_i \oplus N_i \ar[d] \ar[rr]^{(f_{i,1},g_{i,1})}
& & P(i) \ar[d]  \\
A  \ar[rr]^{f_{i,l}}& & B.}
$$
Since the paths from $P(i)$ to $A_i$ and $B_i$ are maximal sectional paths, we have that
both $\Hom_{\Lambda}(N_i, P(i))$ and $\Hom_{\Lambda}(P(i), B)$ are non-zero, but $\Hom_{\Lambda}(N_i, A)=0$.
So the diagram can not commute, and hence $P(i)$ can not almost factor through $f_{i,l}$.
Similarly, $P(i)$ can not almost factor through $g_{i,m}$.
\end{proof}

\begin{proposition}
Let $Q$ be the quiver
$$1\rightarrow2\rightarrow\cdots\rightarrow n-1 \rightarrow n.$$
Then we have
$$\Det(\Lambda)=\{P(i),S(i)\mid 1 \leq i \leq n-1\}\ and \ |\Det(\Lambda)|=2n-2.$$
\end{proposition}

\begin{proof}
First consider the case for epic irreducible morphisms.
We claim that all the almost split sequences of type (2.1) are in the form:
$$0 \to S(i) \to M_i \to S(i-1) \to 0$$
with $2\leq i \leq n$. By [1, Corollary V.4.2], we have that all the non-isomorphic simple modules
$S(i) (1 \leq i \leq n$) belong to the same $\tau$-orbit.
Now assume the almost split sequence starting with $S(i)$ is
$$0 \to S(i) \to M_i \to S(j) \to 0$$
with $M_i$ indecomposable and $j<i$. Because
$$\textbf{dim}S(i)+\textbf{dim}S(j)=\textbf{dim}M_i,$$
we have $j=i-1$. Otherwise, we have $\textbf{dim}M_i=[i,i]+[j,j]$. It implies that $M_i$ is decomposable,
a contradiction. The claim is proved. By Theorem 3.5 and Lemma 2.5, all the minimal
right determiners of all epic irreducible morphisms are $\{S(i)\mid 1 \leq i \leq n-1\}$.

In the case for monic irreducible morphisms, by Theorem 3.5, it suffices to consider all monic irreducible morphisms
in the form $f_i: X_i \rightarrow P(i)$ for $1 \leq i \leq n-1$. Note that here $X_i=\rad P(i)$. So $P(i)$
almost factors through $f_i$, and hence $C(f_i)=P(i)$ for $1 \leq i \leq n-1$ by Theorem 2.4(2) and Remark 3.3.
\end{proof}

Similarly, we have the following

\begin{proposition}
If $Q$ has a unique sink $i$, then
$$|\Det(\Lambda)|=2n-2$$ and $\{P(j)\mid 1\leq i\leq n$ with $j \neq i\}$
are all the projective minimal right determiners.
\end{proposition}


The following result shows that the distribution of the projective minimal right determiners of irreducible morphisms
can determine the orientation of a quiver in some cases.

\begin{corollary}
We have
\begin{enumerate}
\item[(1)] $|\Det(\Lambda)|=2n-2$
and the projective minimal right determiners are $\{P(i)\mid 1 \leq i \leq n-1\}$ if and only if
$Q$ is the quiver
$$1 \rightarrow 2 \rightarrow\cdots\rightarrow n-1 \rightarrow n.$$
\item[(2)] $|\Det(\Lambda)|=2n-2$ and the projective minimal right determiners are $\{P(i)\mid 2 \leq i \leq n\}$
if and only if $Q$ is the quiver
$$1 \leftarrow 2 \leftarrow\cdots\leftarrow n-1 \leftarrow n.$$
\end{enumerate}
\end{corollary}

\begin{proof} The necessity follows from Proposition 3.10. Conversely, by Proposition 3.9
we have that there are no sources and at most one sink in the vertices $\{i\mid 2 \leq i \leq n-1\}$.
If $i$ with $2 \leq i \leq n-1$ is the unique sink, then $P(i)$ is not a minimal right determiner
by Proposition 3.11, a contradiction. The proof is finished.
\end{proof}

The general formula about $|\Det(\Lambda)|$ for path algebras is the following

\begin{theorem}
Let $\Lambda=KQ$ and set $$p:=|\{i\mid i\ \text{is a source in}\ Q \ \text{with}\ 2\leq i\leq n-1\}|.$$
Then we have
$$|\Det(\Lambda)|=
\begin{cases}
2n-2,   &\mbox{if $p=0$;}\\
2n-p-1, &\mbox{if $p\geq 1$.}
\end{cases}$$
\end{theorem}

\begin{proof}
By Theorem 3.5(1), to determine the minimal right determiners of all monic irreducible morphisms, it
suffices to determine the minimal right determiners of all monic irreducible morphisms in the form $f:X \rightarrow P(i)$.

If $p=0$, then there are no sources and at most one sink in the vertices $\{i\mid 2 \leq i \leq n-1\}$.

If there are no sinks in the vertices $\{i\mid 2 \leq i \leq n-1\}$, then the assertion follows from Proposition 3.10.
If $i$ with $2 \leq i \leq n-1$ is the unique sink, then $P(i)$ is not a minimal right determiner
by Proposition 3.9. Note that the subquiver in the Auslander-Reiten quiver of $\mod \Lambda$
consisting of all projectives is as follows.
$$
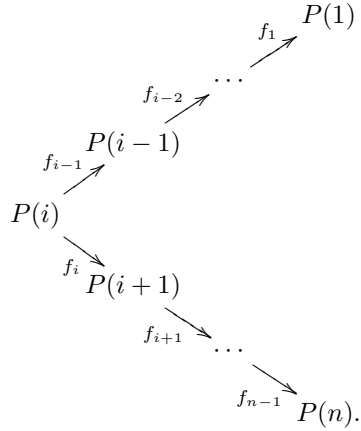
\xymatrix@-13pt{
&&& \makebox[20pt]{${P(1)}$}\\
&& \makebox[20pt]{$\cdots$} \ar[ur]^{f_1} \\
& \makebox[20pt]{${P(i-1)}$}  \ar[ur]^{f_{i-2}} \\
\makebox[20pt]{$P(i)$} \ar[ur]^{f_{i-1}} \ar[dr]_{f_{i}} \\
& \makebox[20pt]{${P(i+1)}$}  \ar[dr]_{f_{i+1}} \\
&& \makebox[20pt]{$\cdots$} \ar[dr]_{f_{n-1}} \\
&&& \makebox[20pt]{${P(n).}$}
}$$
It is easy to check that $C(f_j)=P(j)$ and $C(f_k)=P(k+1)$ for any $1\leq j \leq i-1$ and $i\leq k \leq n-1$.
Now by Corollary 3.7, we have
$$|\Det(\Lambda)|=(n-1)+(i-1)+(n-i)=2n-2.$$

Now let $p\geq1$. Denote the sinks and sources in the quiver $Q$ by $t_l$ and $s_m$ in order respectively.
Take the possible subquivers in the Auslander-Reiten quiver of $\mod \Lambda$ in the following three shapes into account.

(1)
$$
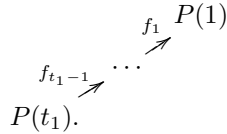
\xymatrix@-18pt{
&& P(1) \\
& \cdots \ar[ur]^{f_1} \\
P(t_1). \ar[ur]^{f_{t_1-1}}}
$$
It is easy to check that $C(f_i)=P(i)$ for any $1\leq i \leq t_1-1$.

(2)
$$
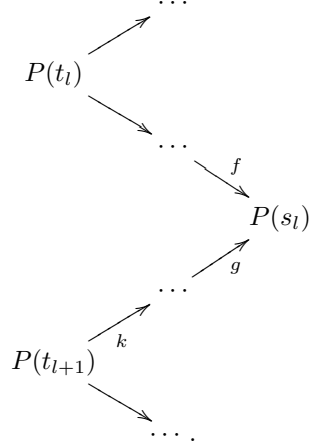
\xymatrix@-10pt{
& \cdots\\
P(t_l)  \ar[dr]\ar[ur]\\
& \cdots \ar[dr]^{f} \\
&& P(s_l) \\
& \cdots \ar[ur]_{g} \\
P(t_{l+1}) \ar[dr]\ar[ur]_k\\
& \cdots.}
$$
For any $0\neq h: P(a) \rightarrow P(b)$, we have $C(f)=P(b)$ except $f$ and $g$.

For $f:P(a)\to P(s_l)$, we have the following commutative diagram
$$
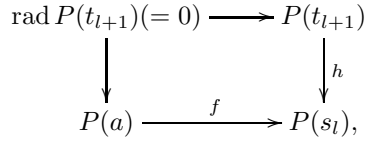
\xymatrix{
\rad P(t_{l+1})(=0) \ar[d] \ar[r]
& P(t_{l+1}) \ar[d]^h  \\
P(a)  \ar[r]^{f}
& P(s_l), }
$$
where $h$ is the composition of monomorphisms from $k$ to $g$.
Thus $C(f)=P(t_{l+1})$ by Theorem 3.5(1). Similarly, we have $C(g)=P(t_l)$.

(3) $$
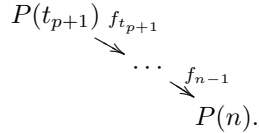
\xymatrix@-18pt{
P(t_{p+1}) \ar[dr]^{\ \ f_{t_{p+1}}} \\
& \cdots \ar[dr]^{f_{n-1}} \\
&& P(n).}
$$
It is easy to check that $C(f_i)=P(i+1)$ for any $t_{p+1}\leq i \leq n-1$.

Therefore, we conclude that
\begin{enumerate}
\item[]
\ \ \ \ $\{$the minimal right determiners of monic irreducible morphisms in $\mod \Lambda\}$\\
$=\{P(i)\mid 1\leq i \leq n$ but $i\neq s_1, \cdots, s_p\}$.
\end{enumerate}
Combining it with Corollary 3.8, we have
$$|\Det(\Lambda)|=(n-1)+(n-p)=2n-p-1.$$
\end{proof}

For an admissible ideal $I$ of $KQ$ and a subquiver $Q'$ of $Q$, we write $I|_{Q'}:=I\cap KQ'$.
To calculate $|\Det(\Lambda)|$ for bound quiver algebras, we introduce the following

\begin{definition}
Assume that $I$ is an admissible ideal of $KQ$ and there exist at least two sinks in $Q$.
Let $s_1<\cdots < s_p$ be the sources in $Q$ with $2 \leq s_m\leq n-1$ for any $1\leq m \leq p$.
For a sink $i$ in $Q$, we define the {\bf sink ideal} $J_i$ of $\Lambda$ as follows.
\begin{enumerate}
\item[(1)] $J_i=I | _{<1,s_1>}$ if $i=1$.
\item[(2)] $J_i=I | _{<s_p,n>}$ if $i=n$.
\item[(3)] $J_i=I | _{<i,s_1>}$ if $i(\neq 1)$ is the first sink and $<1,i>$ is linear.
\item[(4)] $J_i=I | _{<s_p,i>}$ if $i(\neq n)$ is the last sink and $<i,n>$ is linear.
\item[(5)] $J_i=I | _{<s,s'>}$ except the foregoing cases, where $s$, $s'$ are sources and $<s,i>$, $<i,s'>$ are linear.
In this case, we define
\begin{enumerate}
\item[(a)] $J_i \neq 0$ if both $I | _{<s,i>}$ and $I | _{<i,s'>}$ are non-zero;
\item[(b)] $J_i=0$ if at least one of $I | _{<s,i>}$ and $I | _{<i,s'>}$
is zero. In particular, $J_i=0$ if $i=s+1$ or $i=s'-1$.
\end{enumerate}
\end{enumerate}
\end{definition}

Note that any sink in $Q$ exactly belongs to one of these five cases if there exist at least two sinks in $Q$.
Let $Q$ be the quiver
$$\xymatrix@-10pt{1 \ar[r]^{\alpha_1} &2 &3\ar[l]_{\alpha_2}\ar[r]^{\alpha_3} &4 \ar[r]^{\alpha_4}
&5 &6\ar[l]_{\alpha_5} &7\ar[l]_{\alpha_6} &8\ar[l]_{\alpha_7}\ar[r]^{\alpha_8} &9\ar[r]^{\alpha_9}
&10\ar[r]^{\alpha_{10}} &11\ar[r]^{\alpha_{11}} &12\ar[r]^{\alpha_{12}} &13 }
$$
and $I=<\alpha_4 \alpha_3, \alpha_5 \alpha _6 \alpha_7, \alpha_9 \alpha_8, \alpha_{12} \alpha_{11}>$
an admissible ideal of $KQ$. We have that for the sink 2, $J_2= I|_{<2,3>}=0$;
for the sink 5, $J_5= I|_{<3,8>}=<\alpha_4 \alpha_3, \alpha_5 \alpha _6 \alpha_7 >$; and
for the sink 13, $J_{13}= I|_{<8,13>}=<\alpha_9 \alpha_8, \alpha_{12} \alpha_{11}>$.
Moreover, let $Q^{op}$ be the opposite of the above quiver
$$\xymatrix@-10pt{
 1
&2 \ar[l]_{\alpha_1} \ar[r]^{\alpha_2}
&3
&4 \ar[l]_{\alpha_3}
&5 \ar[l]_{\alpha_4} \ar[r]^{\alpha_5}
&6 \ar[r]^{\alpha_6}
&7 \ar[r]^{\alpha_7}
&8
&9 \ar[l]_{\alpha_8}
&10 \ar[l]_{\alpha_9}
&11 \ar[l]_{\alpha_{10}}
&12 \ar[l]_{\alpha_{11}}
&13 \ar[l]_{\alpha_{12}} }
$$
and $I^{op}=<\alpha_3 \alpha_4, \alpha_7 \alpha _6 \alpha_5, \alpha_8 \alpha_9, \alpha_{11} \alpha_{12}>$
an admissible ideal of $KQ^{op}$.
Then we have that for the sink 1, $J_1= I|_{<1,2>}=0$; for the sink 3, $J_3= I|_{<2,5>}=<\alpha_3 \alpha_4 >$;
and for the sink 8, $J_8= I|_{<5,13>}=<\alpha_7 \alpha _6 \alpha_5, \alpha_8 \alpha_9, \alpha_{11} \alpha_{12}>$.

We are now in a position to give the general formula about $|\Det(\Lambda)|$ for bound quiver algebras.

\begin{theorem}
Let $\Lambda=KQ/I$ with $I$ an admissible ideal of $KQ$, and set
$$p:=|\{i\mid i\ \text{is a source in}\ Q \ \text{with}\ 2\leq i\leq n-1\}|,$$
$$q:=|\{i \mid J_i \neq 0\}|,$$
$$r:=|\{i\mid i\ \text{is a sink in}\  Q\ \text{with}\ 1\leq i\leq n\}|.$$
Then we have
$$
|\Det(\Lambda)|=
\begin{cases}
2n-2, &\mbox{if $r=1$;}\\
2n-p-q-1, &\mbox{if $r\geq 2$.}
\end{cases}
$$
\end{theorem}

\begin{proof}
If $r=1$, then the assertion follows from Propositions 3.10 and 3.11.

Now let $r \geq 2$. We claim that $P(i)$ is not the minimal right determiner of any morphism
if and only if either of the following two conditions is satisfied.
\begin{enumerate}
\item[(1)] $i$ is a source with $2 \leq i \leq n-1$.
\item[(2)] $i$ is a sink and $J_i \neq 0$ with $1 \leq i \leq n$.
\end{enumerate}

Let $f_i: \rad P(i) \rightarrow P(i)$ be the inclusion for any $1\leq i\leq n$.
If $i$ is a source with $i=1$ or $n$, then $P(i)=C(f_i)$; if $i$ is a source with $2\leq i \leq n-1$,
then $P(i)$ is not the minimal right determiner of any irreducible morphism by Proposition 3.9.
If $i$ is neither a source nor a sink, then $P(i)=C(f_i)$.

Now let $i$ be a sink and let $s_1<\cdots < s_p$ be the sources with $2 \leq s_m\leq n-1$ for any $1\leq m \leq p$.
We proceed by the classification of sink ideals in Definition 3.14.

(1) Let $i=1$. Then $J_i=I | _{<1,s_1>}$. If $J_i=0$, then there exists a subquiver in the Auslander-Reiten quiver of $\mod \Lambda$
as follows.
$$
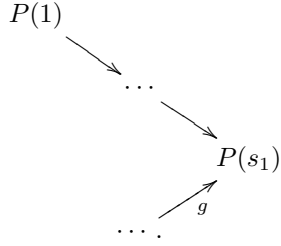
\xymatrix@-10pt{
P(1)  \ar[dr]\\
& \cdots \ar[dr] \\
&& P(s_1) \\
& \cdots.\ar[ur]_{g}}
$$
We have $C(g)=P(1)$. If $J_i \neq 0$, then $P(1)$ is not the minimal right determiner of any irreducible morphism in
$\mod \Lambda$ since $\Hom (P(1), P(s_1))=0$. About the arguments here we refer to that in the proofs of Theorem 3.13
and Proposition 3.11, similarly hereinafter.

(2) Let $i=n$. Then $J_i=I | _{<s_p,n>}$. If $J_i=0$, then there exists a subquiver in the Auslander-Reiten quiver of $\mod \Lambda$
as follows.
$$
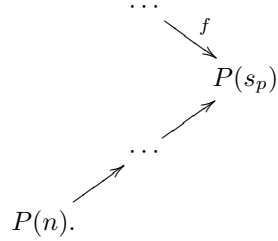
\xymatrix@-10pt{
& \cdots \ar[dr]^{f} \\
&& P(s_p) \\
& \cdots \ar[ur] \\
P(n). \ar[ur]}
$$
We have $C(f)=P(n)$. If $J_i \neq 0$, then $P(n)$ is not the minimal right determiner of any irreducible morphism in $\mod \Lambda$.

(3) Let $i(\neq 1)$ be the first sink and $<1,i>$ linear. Then $J_i=I | _{<i,s_1>}$.
If $J_i=0$, then there exists a subquiver in the Auslander-Reiten quiver of $\mod \Lambda$
as follows.
$$
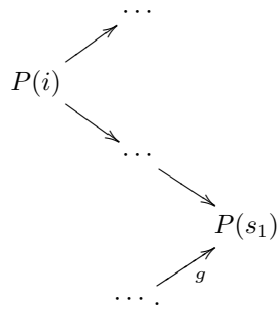
\xymatrix@-10pt{
& \cdots\\
P(i) \ar[ur] \ar[dr]\\
& \cdots \ar[dr] \\
&& P(s_1) \\
& \cdots.\ar[ur]_{g}}
$$
We have $C(g)=P(i)$. If $J_i \neq 0$, then $P(i)$ is not the minimal right determiner of any irreducible morphism in $\mod \Lambda$.

(4) Let $i(\neq n)$ be the last sink and $<i,n>$ linear. Then $J_i=I | _{<s_p,i>}$. If $J_i=0$,
then there exists a subquiver in the Auslander-Reiten quiver of $\mod \Lambda$ as follows.
$$
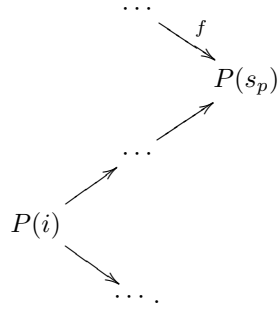
\xymatrix@-10pt{
& \cdots \ar[dr]^f \\
&& P(s_p) \\
& \cdots \ar[ur] \\
P(i) \ar[dr]\ar[ur]\\
& \cdots.}
$$
We have $C(f)=P(i)$. If $J_i \neq 0$, then $P(i)$ is not the minimal right determiner of any irreducible morphism in $\mod \Lambda$.

(5) Finally, let $J_i=I | _{<s,s'>}$ except the foregoing cases, where $s$, $s'$ are sources and $<s,i>$, $<i,s'>$ are linear.
If $J_i=0$, then at least one of $I | _{<s,i>}$ and $I | _{<i,s'>}$ is zero.

If both of $I | _{<s,i>}$ and $I | _{<i,s'>}$ are zero,
then there exists a subquiver in the Auslander-Reiten quiver of $\mod \Lambda$ as follows.
$$\xymatrix@-10pt{
& \cdots \ar[dr]^f \\
&& P(s) \\
& \cdots \ar[ur] \\
P(i) \ar[dr]\ar[ur]\\
& \cdots\ar[dr]\\
& & P(s')\\
& \cdots. \ar[ur]_g}$$
We have $C(f)=P(i)=C(g)$.

If $I | _{<s,i>}=0$ and $I | _{<i,s'>}\neq 0$,
then there exists a subquiver in the Auslander-Reiten quiver of $\mod \Lambda$ as follows.
$$\xymatrix@-10pt{
& \cdots \ar[dr]^f \\
&& P(s) \\
& \cdots \ar[ur] \\
P(i) \ar[dr]\ar[ur]\\
& \cdots.}$$
We have $C(f)=P(i)$.

If $I | _{<s,i>}\neq 0$ and $I | _{<i,s'>}=0$,
then there exists a subquiver in the Auslander-Reiten quiver of $\mod \Lambda$ as follows.
$$\xymatrix@-10pt{
& \cdots \\
P(i) \ar[dr]\ar[ur]\\
& \cdots\ar[dr]\\
& & P(s')\\
& \cdots. \ar[ur]_g}$$
We have $C(g)=P(i)$.

If $J_i\neq 0$, then both of $I | _{<s,i>}$ and $I | _{<i,s'>}$ are non-zero. In this case,
$P(i)$ is not the minimal right determiner of any irreducible morphism in $\mod \Lambda$.

In conclusion, the claim is proved. Then by Theorem 3.5 and Corollary 3.7, we have
$|\Det(\Lambda)|=(n-1)+(n-p-q)=2n-p-q-1$.
\end{proof}

As a consequence of Theorems 3.13 and 3.15, we get the following

\begin{corollary}
For any admissible ideal $I$ of $KQ$, we have
$$|\Det(KQ/I)| \leq |\Det(KQ)|.$$
\end{corollary}

\section{Some examples}

In this section, we give some examples to illustrate the results obtained in Section 3.

\begin{example}
Let $\Lambda=KQ$, where $Q$ is a quiver of type $\mathbb{A}_{2n}$$(n \geq 2)$ with the orientations of the arrows left-and-right
as follows.
$$Q:=1 \buildrel {\alpha_1} \over\longrightarrow 2 \buildrel {\alpha_2}
\over\longleftarrow 3 \buildrel {\alpha_3} \over\longrightarrow 4 \buildrel {\alpha_4}
\over\longleftarrow \cdots \buildrel {\alpha_{2n-2}}
\over\longleftarrow 2n-1 \buildrel {\alpha_{2n-1}} \over\longrightarrow 2n.$$
Then by the proof of Theorem 3.13, we have
\begin{enumerate}
\item[]
\ \ \ \ $\{$the minimal right determiners of monic irreducible morphisms in $\mod \Lambda\}$\\
$=\{P(1),P(2i)\mid 1\leq i \leq n\}$.
\end{enumerate}
Note that the two longest strings here are
$$\alpha_{2n-1} \alpha_{2n-2}^{-1} \cdots \alpha_3 \alpha_2^{-1} \alpha_1,$$
$$\alpha_1^{-1} \alpha_2 \alpha_3^{-1} \cdots \alpha_{2n-2} \alpha_{2n-1}^{-1}.$$
Because $N_{\alpha_1}= \alpha_2^{-1} \alpha_1 \epsilon_1$, we have
$$M(V_{\alpha_1})=M(\epsilon_1)=I(1).$$
Because $N(\alpha_{2i})= \alpha_{2i-1}^{-1} \alpha_{2i} \alpha_{2i+1}^{-1}$, we have
$$M(V_{\alpha_{2i}})=M(\alpha_{2i+1}^{-1})=[2i+1,2i+2],$$ where $1 \leq i \leq n-1$ and $I(2n)=[2n-1,2n]$.
Because $N(\alpha_{2i+1})= \alpha_{2i+2}^{-1} \alpha_{2i+1} \alpha_{2i}^{-1}$, we have
$$M(V_{\alpha_{2i+1}})=M(\alpha_{2i}^{-1})=[2i,2i+1],$$ where $1 \leq i \leq n-2$.
Because $N(\alpha_{2n-1})= \epsilon_{2n} \alpha_{2n-1} \alpha_{2n-2}^{-1}$, we have
$$M(V_{\alpha_{2n-1}})=M(\alpha_{2n-2}^{-1})=[2n-2,2n-1].$$
Then by Corollary 3.7, we have
\begin{enumerate}
\item[]
\ \ \ \ $\{$the minimal right determiners of epic irreducible morphisms in $\mod \Lambda\}$\\
$=\{I(1), [j,j+1], I(2n)\mid 2 \leq j \leq 2n-2\}.$
\end{enumerate}
Therefore we conclude that
$$\Det(\Lambda)=\{P(1),P(2i), I(1), [j,j+1], I(2n)\mid 1 \leq i \leq n\ {\rm and}\ 2 \leq j \leq 2n-2\},$$
and $|\Det(\Lambda)|=3n$.
\end{example}

\begin{example}
Let $Q$ be the quiver
$$\xymatrix@-10pt{1 \ar[r]^{\alpha_1} &2 &3\ar[l]_{\alpha_2}\ar[r]^{\alpha_3} &4 \ar[r]^{\alpha_4} &5
&6\ar[l]_{\alpha_5} &7\ar[l]_{\alpha_6} &8\ar[l]_{\alpha_7}\ar[r]^{\alpha_8} &9\ar[r]^{\alpha_9}
&10\ar[r]^{\alpha_{10}} &11\ar[r]^{\alpha_{11}} &12\ar[r]^{\alpha_{12}} &13 }
$$
and $I=<\alpha_4 \alpha_3, \alpha_5 \alpha _6 \alpha_7, \alpha_9 \alpha_8, \alpha_{12} \alpha_{11}>$
an admissible ideal of $KQ$ as in Section 3.

(1) Let $\Lambda=KQ$. Then by the proof of Theorem 3.13, we have
\begin{enumerate}
\item[]
\ \ \ \ $\{$the minimal right determiners of monic irreducible morphisms in $\mod \Lambda\}$\\
$=\{P(i)\mid 1\leq i \leq 13\ {\rm but}\ i \neq 3,8\}$.
\end{enumerate}
Because $M(V_{\alpha_1})=M(\epsilon_1)=S(1)$, $M(V_{\alpha_2})=[3,5]$,
$M(V_{\alpha_3})=[2,3]$, $M(V_{\alpha_4})=S(4)$, $M(V_{\alpha_5})=S(6)$, $M(V_{\alpha_6})=S(7)$,
$M(V_{\alpha_7})=[8,13]$, $M(V_{\alpha_8})=[5,8]$, $M(V_{\alpha_9})=S(9)$,
$M(V_{\alpha_{10}})=S(10)$, $M(V_{\alpha_{11}})=S(11)$ and $M(V_{\alpha_{12}})=S(12)$, by Corollary 3.7 we have
\begin{enumerate}
\item[]
\ \ \ \ $\{$the minimal right determiners of epic irreducible morphisms in $\mod \Lambda\}$\\
$=\{S(1), [3,5], [2,3], S(4), S(6), S(7), [8,13], [5,8], S(9), S(10), S(11),S(12)\}.$
\end{enumerate}
Therefore we conclude that
$$\Det(\Lambda)=\{P(i), S(1), S(4), S(6), S(7), S(9), S(10), S(11),S(12),$$
$$\ \ \ \ \ \ \ \ \ [3,5], [2,3], [8,13], [5,8] \mid 1\leq i \leq 13\ {\rm but}\ i \neq 3,8\}$$
and $|\Det(\Lambda)|=23$.

(2) Let $\Lambda=KQ/I$. Then by the proof of Theorem 3.15, we have
\begin{enumerate}
\item[]
\ \ \ \ $\{$the minimal right determiners of monic irreducible morphisms in $\mod \Lambda\}$\\
$=\{P(i)\mid 1\leq i \leq 13\ {\rm but}\ i \neq 3,5,8,13\}$.
\end{enumerate}
Because $M(V_{\alpha_1})=M(\epsilon_1)=S(1)$, $M(V_{\alpha_2})=[3,4]$, $M(V_{\alpha_3})=[2,3]$, $M(V_{\alpha_4})=S(4)$,
$M(V_{\alpha_5})=S(6)$, $M(V_{\alpha_6})=S(7)$, $M(V_{\alpha_7})=[8,9]$, $M(V_{\alpha_8})=[6,8]$, $M(V_{\alpha_9})=S(9)$,
$M(V_{\alpha_{10}})=S(10)$, $M(V_{\alpha_{11}})=S(11)$ and $M(V_{\alpha_{12}})=S(12)$,
by Corollary 3.7 we have
\begin{enumerate}
\item[]
\ \ \ \ $\{$the minimal right determiners of epic irreducible morphisms in $\mod \Lambda\}$\\
$=\{S(1), [3,4], [2,3], S(4), S(6), S(7), [8,9], [6,8], S(9), S(10), S(11), S(12)\}.$
\end{enumerate}
Therefore we conclude that
$$\Det(\Lambda)=\{P(i), S(1), S(4), S(6), S(7), S(9), S(10), S(11),S(12),$$
$$\ \ \ \ \ \ \ \ \ \ \ \ \ \ [3,4], [2,3], [8,9], [6,8] \mid 1\leq i \leq 13\ {\rm but}\ i \neq 3,5,8,13\}$$
and $|\Det(\Lambda)|=21$.
\end{example}

\bigskip

{\bf Acknowledgement.}
This research was partially supported by NSFC (Grant No. 11571164) and a Project Funded by the Priority
Academic Program Development of Jiangsu Higher Education Institutions. The authors would like to thank
Claus M. Ringel for his helpful suggestions on the result about string algebras and the number of almost
split sequences of type (2.1). The authors also thank the referee for the useful comments.

\end{document}